\documentclass[12pt]{article}
\usepackage{amsmath}
\usepackage{amssymb}%
\usepackage{amsthm}
\newtheorem{theo}{Theorem}[section]
\newtheorem{theorem}{Theorem}[section]
\newtheorem{lem}[theorem]{Lemma}
\newtheorem{lemma}[theorem]{Lemma}
\newtheorem{cor}[theorem]{Corollary}

\newtheorem{proposition}[theorem]{Proposition}

\newtheorem{rem}[theorem]{Remark}

\newtheorem{defi}[theorem]{Definition}

\theoremstyle{definition}
\newtheorem{ex}[theorem]{Example}

\renewcommand{\P}{\mathbf P}
\newcommand{\E}{\mathbf E}
\newcommand{\R}{\mathbb{R}}
\newcommand{\N}{\mathbb{N}}

\newcommand{\Z}{\mathbb{Z}}

\newcommand{\Ha}{\mathbb{H}}
\newcommand{\eps}{\varepsilon}
\newcommand{\Var}{\mathrm{Var}}

\newcommand{\Cov}{\mathrm{Cov}}

\begin{document}

\title{Extreme-Value Analysis of Standardized Gaussian Increments}
\author{Zakhar Kabluchko}
\maketitle
\begin{center}
Institut f\"ur Mathematische Stochastik\\
Georg-August-Universit\"at G\"ottingen\\
Maschm\"uhlenweg 8-10\\ 
D-37073 G\"ottingen\\
E-mail: kabluch@math.uni-goettingen.de
\end{center}

\begin{abstract}
\noindent
Let $\{X_i,i=1,2,\ldots\}$ be i.i.d. standard gaussian variables. Let $S_n=X_1+\ldots+X_n$ be the sequence of  partial sums and
$$
L_n=\max_{0\leq i<j\leq n}\frac{S_j-S_i}{\sqrt{j-i\,}}.
$$
We show that the distribution of $L_n$, appropriately normalized, converges as $n\to\infty$ to the Gumbel distribution. In some sense, the the random variable $L_n$, being the maximum of $n(n+1)/2$ dependent standard gaussian variables, behaves like the maximum of $Hn \log n$ independent standard gaussian variables. Here, $H\in (0,\infty)$ is some constant. We also prove a version of the above result for the Brownian motion. 
\end{abstract}

\section{Introduction}
A basic result in extreme-value theory 
says that if $\{X_i,i\in\N\}$ are independent standard normal random variables, then the distribution of  $M_n=\max\{X_1,\ldots,X_n\}$ converges, after appropriate normalization, to the Gumbel law. More precisely, let 
\begin{equation}\label{eq:defab}
a_n=\sqrt{2\log n}+\frac{-1/2 \log\log n -\log 2\sqrt{\pi}  }{\sqrt{2\log n}},\qquad b_n=\frac{1}{\sqrt{2\log n}}.
\end{equation}
Then, for every $\tau\in\R$,
\begin{equation}\label{eq:equatmain}
\lim_{n\to\infty}\P\left [M_n\leq a_n+b_n\tau \right ]=\exp(-e^{-\tau}).
\end{equation}
It is also well known that the above result remains true for dependent gaussian variables if the dependence is weak enough. We mention only one example, due to Berman (see~\cite[Chapter 4]{Lead}). Let $\{X_i, i\in\N \}$ be a stationary centered gaussian sequence with constant variance $1$ such that  the covariance function $r(n)=\Cov(X_1,X_n)$ satisfies $r(n)=o(1/\log n)$ as $n\to\infty$. Then~\eqref{eq:equatmain} holds with the same normalizing constants.

An example of a situation where the dependence can not be ignored is given by the Darling-Erd\"os theorem~\cite{DarlErd}. 
\begin{theorem}\label{theo:DarlErd}
Let $\{X_i,i\in\N\}$ be i.i.d. standard normal variables. Define $S_n=X_1+\ldots+X_n$ and  let
$$
M_n=\max_{k\in\{1,\ldots,n\}} \frac {S_k}{\sqrt k}.
$$
Then, for every $\tau\in\R$,
$$
\P\left[M_n\leq a_n+b_n\tau \right]\to \exp(-e^{-\tau}),
$$
where 
$$
a_n=\sqrt{2\log\log n}+\frac {1/2 \log\log\log n-\log 2\sqrt \pi}{\sqrt{2 \log\log n}},\qquad
b_n=\frac {1}{\sqrt{2\log\log n}}.
$$
\end{theorem}

The next theorem, together with a strong approximation argument, was used by Darling and Erd\"os  to prove Theorem~\ref{theo:DarlErd}.
\begin{theorem}\label{theo:DarlErdBrown}
Let $\{B(x),x\geq 0\}$ be the standard Brownian motion. For $n>1$ define
$$
M_n=\sup_{x\in[1,n]}\frac{B(x)}{\sqrt x}.
$$
Then, for every $\tau\in\R$,
$$
\P\left[M_n\leq a_n+b_n\tau \right]\to \exp(-e^{-\tau}),
$$
where the normalizing constants are the same as in the previous theorem.
\end{theorem}
Theorem~\ref{theo:DarlErdBrown} may be viewed as a \textit{distributional convergence} version of the law of the iterated logarithm. In somewhat unusual form (see Theorem~14.15 in~\cite{Rev}), the law of the iterated logarithm states that, almost surely,
$$
\lim_{n\to\infty} \frac{1}{\sqrt{2\log\log n}}\sup_{x\in [1,n]}\frac{B(x)}{\sqrt x}=1.
$$
See~\cite{Kh} for another distributional convergence version of the law of the iterated logarithm.

Of course, the Darling-Erd\"os theorem is true not only for standard normal variables. A necessary and sufficient condition on the distribution of the  i.i.d. variables $X_i$ for  the Darling-Erd\"os theorem to hold was found by Einmahl~\cite{Einmahl}. Bertoin~\cite{Bert} proved an analog of the Darling-Erd\"os theorem for random variables with distributions attracted to  stable laws.

The next theorem is the main result of this paper. 
\begin{theorem}\label{theo:maindiscr}
Let $\{X_i,i\in\N\}$ be i.i.d. standard normal random variables. Define $S_n=X_1+\ldots+X_n$ and $S_0=0$. Let
$$
L_n=\max_{0\leq i<j\leq n}\frac{S_j-S_i}{\sqrt{j-i\,}}.
$$
Then, for every $\tau\in\R$,
$$
\lim_{n\to\infty}\P\left [L_n\leq a_n+b_n\tau \right ]=\exp(-e^{-\tau}),
$$
where  $a_n$ and $b_n$ are given by
\begin{equation}\label{eq:maindiscr}
a_n=\sqrt{2\log n}+\frac {1/2 \log\log n+\log H-\log 2\sqrt{\pi}}{\sqrt{2 \log n}},\qquad
b_n=\frac {1}{\sqrt{2\log n}}
\end{equation}
for some constant $H\in(0,\infty)$.
\end{theorem}
The constant  $H$ is defined as follows. Let $\{B(t),t\geq 0\}$ be the standard Brownian motion. Let
$$
F(a)=\lim_{T\to\infty} \frac 1 T\E \left[\exp \sup_{t\in [0,T]\cap a\Z} (B(t)-t/2)\right]
$$
and
\begin{equation}\label{eq:defg1}
G(y)=\frac{1}{y^2}F\left(\frac{2}{y}\right)^2.
\end{equation}
Then $H=4\int_{0}^{\infty} G(y)dy$. A more explicit formula for $H$ will be given later in Section~\ref{sec:const_H}.

The motivation for studying the distribution of $L_n$ was the fact that $L_n$ as well as related quantities are of interest in statistics~\cite{Dav, DSp}.

The question about the asymptotic distribution of $L_n$ was studied by Huo~\cite{Huo0}, \cite{Huo}. Note, however, that his result does not imply Theorem~\ref{theo:maindiscr}. In particular, the normalizing constants in \cite{Huo} differ from the values given in~\eqref{eq:maindiscr}
and are, in fact, random variables.\footnote[1]{After the second version of this paper was submitted to arXiv, the author became aware that  Theorem~\ref{theo:maindiscr} was proved in D. Siegmund, E. S. Venkatraman.
Using the generalized likelihood ratio statistic for sequential detection of a change-point.  
Ann. Statist. 23(1995), 255-271. For a related result see also D. Siegmund, B. Yakir.
Tail probabilities for the null distribution of scanning statistics. 
Bernoulli 6(2000),  191-213.}

The next theorem describes the  almost sure limiting  behavior of $L_n$.  It is a consequence of a  more general result due to Shao~\cite{Shao}, who proved a conjecture of  R\'ev\'esz~\cite[\S 14.3]{Rev} (see also \cite{Stei} for a simplification of Shao's proof and~\cite{Lanz} for a related result).  
\begin{theorem}\label{theo:mainalmostsure}
With the notation of Theorem~\ref{theo:maindiscr} we have, almost surely,
$$
\lim_{n\to\infty}\frac{L_n}{\sqrt{2\log n}}=1.
$$
\end{theorem}


The next theorem may be viewed as a distributional convergence version of the Erd\"os-Renyi law of large numbers in the case of standard normal summands and is a consequence of a more general result of Koml\'os and Tusn\'ady proved in~\cite{Komlos} (see also~\cite{PitKoz, Zholud1, Zholud2}). We give a short proof of this theorem in Section~\ref{sec:prooferdren}.
\begin{theorem}\label{theo:mainerdren}
Let $\{X_i,i\in\N\}$ be i.i.d. standard normal random variables. Fix some $c>0$ and let $l_n=[c\log n]$. Define $S_n=X_1+\ldots+X_n$ and let
$$
L_{n,c}= \frac 1 {\sqrt {l_n}} \sup_{0\leq k\leq n-l_n}   (S_{k+l_n}-S_{k}).
$$
Then, for every $\tau\in\R$,
$$
\lim_{n\to\infty}\P\left [L_{n,c}\leq a_n+b_n\tau \right ]=\exp(-e^{-\tau}),
$$
where the constants $a_n$ and $b_n$ are given by
$$
a_n=\sqrt{2\log n}+\frac {-1/2 \log\log n+\log((4/c)F(4/c))-\log2\sqrt{\pi}}{\sqrt{2 \log n}},\,
b_n=\frac {1}{\sqrt{2\log n}}.
$$
\end{theorem}

We also prove the following continuous counterpart of Theorem~\ref{theo:maindiscr}.
\begin{theorem}\label{theo:main}
Let $\{B(x),x\geq 0\}$ be the standard Brownian motion. For $n>1$ define
$$
L_n=\sup_{\genfrac{}{}{0pt}{1} {x_1,x_2\in[0,1]} {x_2-x_1\geq 1/n}} \frac {B(x_2)-B(x_1)}{\sqrt{x_2-x_1}}.
$$
Then, for every $\tau\in\R$,
$$
\lim_{n\to\infty}\P\left [L_n\leq a_n+b_n\tau \right ]=\exp(-e^{-\tau}),
$$
where the constants $a_n$ and $b_n$ are given by
$$
a_n=\sqrt{2\log n}+\frac {3/2 \log\log n-\log 2\sqrt \pi}{\sqrt{2 \log n}},\qquad
b_n=\frac {1}{\sqrt{2\log n}}.
$$
\end{theorem}
Recall that a classical theorem of L\'evy on the  modulus of continuity of Brownian sample paths (see e.g. \cite{ItoMcKean}) asserts  that, almost surely,
$$
\limsup_{n\to \infty}\frac{1}{\sqrt{2\log n}}\sup_{\genfrac{}{}{0pt}{1} {x_1,x_2\in[0,1]} {x_2-x_1=1/n}}\frac {B(x_2)-B(x_1)}{\sqrt{x_2-x_1}}=1.
$$
It is not difficult to deduce from this that
$$
\lim_{n\to \infty}\frac{1}{\sqrt{2\log n}}\sup_{\genfrac{}{}{0pt}{1} {x_1,x_2\in[0,1]} {x_2-x_1\geq 1/n}}\frac {B(x_2)-B(x_1)}{\sqrt{x_2-x_1}}=1.
$$
Thus, Theorem~\ref{theo:main} may be viewed  as a distributional convergence  version of L\'evy's modulus of continuity.

Since the normalizing constants in Theorems~\ref{theo:maindiscr} and~\ref{theo:main} are different, it seems to be impossible to deduce 
Theorem~\ref{theo:maindiscr} from its continuous counterpart Theorem~\ref{theo:main} by a strong approximation argument as it was done by Darling and Erd\"os in their proof of Theorem~\ref{theo:DarlErd}.   


\section{Asymptotic Extreme-Value Rate}
In this section we are going to introduce the notion of asymptotic extreme-value rate, which will allow us to compare the results of Theorems~\ref{theo:DarlErd},~\ref{theo:DarlErdBrown}, \ref{theo:maindiscr}, \ref{theo:mainerdren},~\ref{theo:main} with the classical extreme-value theorem for i.i.d. normal variables stated at the beginning of the paper. 
Let $\{\xi_i, i=1,\ldots,N\}$ and $\{\eta_i,i=1,\ldots,N\}$ be two jointly gaussian vectors. We suppose that the variables $\xi_i$ and $\eta_i$ are centered and have variance $1$. Suppose, moreover, that the variables $\eta_i$ are independent, whereas $\xi_i$ are not. Then it is well known that, in some sense, $\max_{i=1,\ldots,N} \xi_i$ is dominated by $\max_{i=1,\ldots,N} \eta_i$. One way to make this  claim precise is the Slepian Comparison Lemma (see e.g.~\cite[Corollary 4.2.3]{Lead}) which states that, for every $u$, 
$$
\P\left[\max_{i=1,\ldots,N} \xi_i>u\right]\leq \P\left[\max_{i=1,\ldots,N} \eta_i>u\right].
$$
Given a dependent vector $\{\xi_i,i=1,\ldots, N\}$ of standard normal variables, we would like to determine the number $f(N)$ of \textit{independent} standard normal variables $\{\eta_i, i=1,\ldots, f(N)\}$ such that  behavior of $\max_{i=1,\ldots,f(N)}\eta_i$ is in some sense close to the behavior of the maximum of the dependent vector $\xi_i$. By the above, we should have $f(N)\leq N$. The next definition makes this precise.
\begin{defi}
For each $n\in \N$ let a  gaussian field  $\{\xi_n(t), t\in T_n\}$ defined on some parameter space $T_n$ be given. Suppose that for all $n$ the field $\xi_n$ is centered and has constant variance $1$. Let $f:\N\to\R$ be some function. We say that the sequence $\xi_n$ has asymptotic extreme-value rate $f$ if, for each $\tau\in \R$,
$$
\lim_{n\to\infty}\P\left[\sup_{t\in T_n} \xi_n(t)\leq a_{f(n)}+b_{f(n)}\tau \right]=\exp(-e^{-\tau}),
$$
where $a_n$ and $b_n$ are constants defined in~\eqref{eq:defab}.
\end{defi}
Thus, the sequence of gaussian fields $\xi_n$ is said to have asymptotic extreme-value rate $f$ if, for large $n$, the supremum of $\xi_n$ has the same behavior as the supremum of $f(n)$ i.i.d. standard normal variables.

Now we are going to  compute the extreme-value rates of  gaussian fields defined in Theorems~\ref{theo:DarlErd},~\ref{theo:DarlErdBrown},~\ref{theo:maindiscr},~\ref{theo:mainerdren},~\ref{theo:main}. To this end,  we need two simple lemmas. The first one can be proved by a simple calculation. For the second lemma, which is due to Khintchine,  see e.g.~\cite[Theorem 1.2.3]{Lead}.
\begin{lemma}
Let the constants $a_n,b_n$ be defined by~\eqref{eq:defab} and let $f(n)=cn(\log n)^b$. Then, as $n\to\infty$,
$$
a_{f(n)}=\sqrt{2\log n}+\frac{(-1/2+b) \log\log n+\log c -\log 2\sqrt{\pi}  }{\sqrt{2\log n}}+o\left(\frac{1}{\sqrt{2\log n}}\right),
$$
$$
b_{f(n)}\sim\frac{1}{\sqrt{2\log n}}.
$$
\end{lemma}
\begin{lemma}
Let $M_n$ be a sequence of random variables such that, for some constants $a_n',b_n'$, the distribution of $(M_n-a_n')/b_n'$ converges as $n\to\infty$ to some  non-degenerate distribution function $G$. Let another constants $a_n''$, $b_n''$ be given and suppose that
$$
\lim_{n\to\infty} b_n'/b_n''=1,\qquad \lim_{n\to\infty} (a_n'-a_n'')/b_n'=1.
$$
Then the distribution of $(M_n-a_n'')/b_n''$ converges to $G$ as well.
\end{lemma}
Using the above two lemmas, one deduces easily that the gaussian fields considered in Theorems~\ref{theo:DarlErd},~\ref{theo:DarlErdBrown},~\ref{theo:maindiscr},~\ref{theo:mainerdren},~\ref{theo:main} have asymptotic extreme-value rates given in the following table. The usual notation is used, i.e. $\{X_k,k\in\N\}$ is a sequence of i.i.d. standard normal variables, $S_n=X_1+\ldots+X_n$ are the partial sums and $\{B(x),x\geq 0\}$ is the standard Brownian motion. 
\vspace*{4mm}

\begin{tabular}{|l|l|l|l|}
\hline
&$T_n$& $\xi_n$& $f(n)$ \rule{0mm}{5mm}\\
\hline 
1.&\rule{0mm}{5mm}$\{1,\ldots,n\}$ & $\xi_n(k)=X_k$& $n$ \rule{0mm}{3mm}\\
2.&$\{1,\ldots,n\}$ & $\xi_n(k)=\frac{S_k}{\sqrt k}$ & $\log n \log\log n$\rule{0mm}{5mm}\\
3.&$[1,n]$ & $\xi_n(x)=\frac{B(x)}{\sqrt x}$ & $\log n \log\log n$\rule{0mm}{5mm}\\
4.&$\{(i,j)\;|\; 0\leq i<j\leq n\}$ & $\xi_n(i,j)=\frac{S_j-S_i}{\sqrt{j-i}}$ & $Hn\log n$\rule{0mm}{5mm}\\
5.&$\left\{(x_1,x_2): \begin{matrix}&x_1,x_2\in[0,1]\\& x_2-x_1\geq 1/n \end{matrix}\right\}$  
& $\xi_n(x_1,x_2)=\frac{B(x_2)-B(x_1)}{\sqrt{x_2-x_1}}$ & $n(\log^2n)$\rule{0mm}{5mm} \\
6.&$\{0,1,\ldots, n-[c\log n]\}$ & $\xi_n(k)=\frac{S_{k+[c\log n]}-S_k}{\sqrt{[c\log n]}}$ & $(4/c)F(4/c)n$\rule{0mm}{5mm}\\
7. &$\left\{(x_1,x_2): \begin{matrix}&x_1,x_2\in[0,1]\\&x_2-x_1=1/n\end{matrix} \right\}$ & $\xi_n(x_1,x_2)=\frac{B(x_2)-B(x_1)}{\sqrt{x_2-x_1}}$ & $n\log n$ \rule{0mm}{5mm}\\
\hline 
\end{tabular}

\vspace*{5mm}
Note that entry $7$ can be easily deduced from Pickands' results~\cite{P1}(or see~\cite[Chapter 12]{Lead}).
 
It is a priori clear that the asymptotic rate of entry $2$ in the above table should not be faster than the rate of entry $3$. The reason is that the distribution of $\{S_k/\sqrt{k}, k=1,\ldots,n\}$ may be identified with the distribution of $\{B(k)/\sqrt{k}, k=1,\ldots,n\}$. In fact, as Darling and Erd\"os showed, the rates in entry $2$ and entry $3$ are equal.  
Similarly, there is an embedding of the gaussian vector from the entry $4$ into the process from the entry $5$, namely one can identify $\{(S_j-S_i)/\sqrt{j-i},0\leq i<j\leq n\}$ with $\{(B(j/n)-B(i/n))/\sqrt{(j-i)/n}, 0\leq i<j\leq n\}$. Thus, it is clear that the rate of entry $4$ is not faster than that of entry $5$. A somewhat surprising fact is that these rates do not coincide.  

The rest  of the paper is organized as follows. 
In Section \ref{sec:locstat} we recall the definition of locally stationary gaussian fields. 
The main results of this section are Corollary~\ref{theo:tail} and Corollary~\ref{cor:1}.  
In Section \ref{sec:proofbrown} we prove Theorem~\ref{theo:main}. The main tools are Corollary~\ref{theo:tail} and Berman's inequality.  The proof of Theorem \ref{theo:mainerdren} is given in Section~\ref{sec:prooferdren}. Finally, Section~\ref{sec:proofdiscr} is devoted to the proof of Theorem~\ref{theo:maindiscr}.

\section{Locally Stationary Gaussian Fields}\label{sec:locstat}
Given a centered gaussian field $\{X(t),t\in \R^d\}$ with constant variance $1$ we would like to obtain an exact asymptotics of the so-called high excursion probability of $X$ over a given compact set $K$, i.e.  a result of the form
\begin{equation}\label{eq:highex}
\P\left[\sup_{t\in K}X(t)>u\right]\sim C_K u^D e^{-u^2/2}, \qquad u\to\infty 
\end{equation}
for a number $D$ depending on the structure of the field and a constant $C_K$ depending on the set $K\subset \R^d$ and the structure of the field. 

After preliminary results by Cramer, Leadbetter, Volkonski, Rozanov, Ber\-man, Slepian  and others, this question was studied by Pickands \cite{P1,P2} (see also \cite[Chapter 12]{Lead}, \cite{Pit}, \cite{Piterbarg}). To state his result, let $\{X(t),t\in \R\}$ be a stationary centered gaussian process whose covariance function $r(s)=\E[X(0)X(s)]$ satisfies 
$$
r(s)=1-C|s|^{\alpha}+o(|s|^{\alpha}),\qquad s\to 0
$$  
for some $\alpha\in(0,2]$, called the index of the process $X$, and some $C>0$. Suppose also that $r(s)=1$ holds only for $s=0$.  Under these conditions, Pickands proved the asymptotic equality
$$
\P\left[\sup_{t\in [0,l]} X(t)>u\right]\sim l H_{\alpha}C^{1/\alpha} \frac{1}{\sqrt{2\pi}}u^{2/\alpha-1} e^{-u^2/2}, \qquad u\to\infty, 
$$
where $H_{\alpha}\in (0,\infty)$ is some constant. Only the values $H_{1}=1$ and $H_2=1/{\sqrt{\pi}}$ are known rigorously. There is a conjecture that $H_{\alpha}=1/\Gamma(1/\alpha)$ (see~\cite{Burnecki}).

Pickands' result was generalized by Qualls and Watanabe \cite{QuallsWatanabe,QuallsWatanabe1}, who allowed a slightly more general class of covariance functions and considered isotropic fields defined on the $d$-dimensional euclidian space; by Bickel and Rosenblatt~\cite{Bickel}, who considered two-dimensional stationary fields;  by Albin~\cite{Albin3}, who considered non-gaussian stationary processes,  as well as by many others.    
However in  this paper, we need an estimate of the form~\eqref{eq:highex} for non-stationary gaussian fields. 
On a heuristical level, Aldous~\cite{Ald} applied his method of Poisson clumping heuristic, which is close to Pickands' method, to many non-stationary fields. 
In~\cite{Huesler},  H\"usler applied Pickands' methods to study the high excursion probability for  non-stationary centered gaussian processes defined on the real line with covariance function $r(t_1,t_2)=\E[X(t_1)X(t_2)]$ satisfying  
$$
r(t,t+s)=1-C(t)|s|^{\alpha}+o(|s|^{\alpha}),\qquad s\to 0
$$
uniformly on compacts in $t$ for some continuous function $C(t)>0$.  H\"usler calls such processes \textit{locally stationary}. It should be noted that not every  stationary process is locally stationary. H\"usler proves that, as $u\to\infty$,
$$
\P\left[\sup_{t\in [l_1,l_2]} X(t)>u\right]\sim H_{\alpha} \left(\int_{l_1}^{l_2}C^{1/\alpha}(t)dt\right) \frac{1}{\sqrt{2\pi}} u^{2/\alpha-1} e^{-u^2/2}. 
$$
Thus, the function $C^{1/\alpha}(t)$ may be thought of as a sort of  intensity measuring the contribution of the point $t$ to the high excursion probability.

The notion of locally stationary processes was extended to  fields defined on the $d$-dimensional euclidian space (or, even more generally, on compact manifolds)  by Mikhaleva and Piterbarg in~\cite{PM} and by Chan and Lai in~\cite{ChL}. 


First we recall the definition of homogeneous functions. 
\begin{defi}
A function $f:\mathbb R^d \to \mathbb R$  is called homogeneous of order $\alpha>0$ if for each
$s\in \R^d$ and $\lambda\in \R$
$$
f(\lambda s)=|\lambda|^\alpha f(s).
$$
\end{defi}
In particular, homogeneous functions are symmetric, i.e. they satisfy $f(s)=f(-s)$.
Let $H(\alpha)$ be the set of all continuous homogeneous functions of order~$\alpha$.
For $f\in H(\alpha)$ define $\|f\|=\sup_{\|t\|_2=1}f(t)$. With this norm, $H(\alpha)$
is a Banach space which can be identified with the space $C(\mathbb S^{d-1})$
of continuous functions on the unit sphere in $\mathbb R^d$.\\
Let $H^+(\alpha)$ be the cone of all strictly positive functions in $H(\alpha)$.

Now we are ready to define  locally stationary gaussian fields.
\begin{defi}[see \cite{ChL}]
Let $\{X(t),t\in D\}$ be a centered gaussian field with constant variance $1$ defined on some domain  $D\subset \mathbb R^d$. Let $r(t_1,t_2)=\E[X(t_1)X(t_2)]$ be the covariance function of $X$ and suppose that it satisfies $r(t_1,t_2)=1 \Leftrightarrow t_1=t_2$. 
The field $X$ is called locally stationary with index  $\alpha\in(0,2]$  if for each $t\in D$ a continuous function $C_t\in H^+(\alpha)$ exists such that the following conditions hold
\begin{itemize}
\item[1. ] We have
$$
\lim_{\|\,s\|_2 \to 0} \frac {1-r(t,t+s)}{C_t(s)}=1
$$
uniformly on compacts.
\item[2. ] The map $C_{\bullet}:D\to H^+(\alpha)$, sending $t$ to $C_t$, is continuous.
\end{itemize} 
\end{defi}
The collection of homogeneous functions $C_{t}$ is referred to as \textit{the local structure} of the field~$X$. 

The next proposition gives a representation for the local structure of a locally stationary field. Note that it differs from the  corresponding representation in \cite{PM}.
\begin{proposition}\label{prop:existtang} Let $\{X(t), t\in D\}$ be a locally stationary gaussian field  of index $\alpha$ with local structure $C_t(s)$.  Then, for each fixed $t\in D$, the function $C_t(\cdot)$ is negative definite. Moreover, there exists a finite measure $\Gamma_t$ on $\mathbb S^{d-1}$ such that the following representation holds
$$
C_t(s)=\int_{\mathbb S^{d-1}}|(s,x)|^{\alpha} d\Gamma_t(x).
$$
The support of $\Gamma_t$ is not contained in any proper linear subspace of $\mathbb R^d$. 
\end{proposition}
\begin{proof}
Recall (see e.g. \cite[p.74]{Berg}) that a continuous function $f:\R^d\to \R$ satisfying $f(s)=f(-s)$ and $f(0)=0$ is called \textit{negative definite} if  for each $s_1,\ldots,s_n\in \R^d$ the matrix 
$$
\left(f(s_i)+f(s_j)-f(s_i-s_j)\right)_{i,j=1,\ldots,n}
$$ 
is positive definite. 
For $u>0$ set $q=q(u)=u^{-2/\alpha}$. Define the gaussian vector $\{Y_i=Y_i(u),i=1,\ldots,n\}$ by 
$$
Y_i=u(X(t+qs_i)-u).
$$
Consider the joint distribution of $\{Y_i,i=1,\ldots,n\}$  conditioned on $X(t)=u$. It is (non-centered) gaussian and the well-known formulas for the conditional gaussian distributions show that its covariance matrix is 
$$
\left(u^2r(t+qs_i,t+qs_j)-u^2r(t,t+qs_i)r(t,t+qs_j) \right)_{i,j=1,\ldots,n}.
$$ 
It follows from the definition of  local stationarity that, as $u\to\infty$, this converges to
$$
\left(C_t(s_i)+C_t(s_j)-C_t(s_i-s_j)\right)_{i,j=1,\ldots,n}.
$$
Since the above matrix is positive definite as a limit of positive definite matrices, it follows that the function $C_t(\cdot)$ is  negative definite for each $t$. 

By Schoenberg's  theorem (see e.g. \cite[Theorem 2.2]{Berg}) the function $\exp(-C_t(\cdot))$ is positive definite and thus is the characteristic function of some symmetric probability measure $\mu_t$ on $\R^d$. Since $C_t(\cdot)$ is homogeneous of order $\alpha$, the measure $\mu_t$ is stable of order $\alpha$. The remaining part of the proposition follows from the classification of symmetric stable measures on $\R^d$ (see e.g. \cite[Theorem 2.4.3]{samtaqq}).
\end{proof}

Now we give some examples of locally stationary fields.
\begin{ex}[see \cite{P1}]\label{ex:slepian}
Let $\{X(t),t\in\R\}$ be a centered stationary gaussian process with constant variance $1$. Suppose that the covariance function $r(t)=\E[X(0)X(t)]$ satisfies the Pickands condition
$$
r(s)=1-C|\,s|^{\alpha}+o(|\,s|^{\alpha}), \qquad s\to 0
$$
for some $C>0$ and $\alpha\in(0,2]$. Then $X$ is locally stationary of index $\alpha$. The local structure is given by $C_t(s)=C$. Examples include, to mention only  a few,  $r(t)=\exp(-|\,t|^{\alpha})$ (the generalized Ornstein-Uhlenbeck process),  $r(t)=(1+|\,t|^{\alpha})^{-\beta}$ for  $\alpha\in (0,2]$ and  $\beta>0$ (the generalized Cauchy model, see e.g.~\cite{Schlather}), $r(t)=\max(1-|\,t|,0)$ (the Slepian process). In the latter case,  $\alpha=1$.
\end{ex}
\begin{ex}[see \cite{Ald}]
Let $\{B(t),t\geq 0\}$ be the standard Brownian motion. The \textit{standardized Brownian motion} is the process $\{X(t),t>0\}$ defined by 
$$
X(t)=B(t)/ \sqrt t. 
$$
The standardized Brownian motion is locally stationary with index $\alpha=1$.
The local structure is given by $C_t(s)=\frac {|s|}{2t}$.
\end{ex}
\begin{proof}
Using that $\Cov(B(t_1),B(t_2))=\min(t_1,t_2)$ we obtain, for $s>0$,  
$$
r(t,t+s)=\Cov(X(t),X(t+s))=\frac{t}{\sqrt{t(t+s)}}=1-\frac {s}{2t}+O(s^2).
$$
For $s<0$ we obtain
$$
r(t,t+s)=\Cov(X(t),X(t+s))=\frac{t+s}{\sqrt{t(t+s)}}=1+\frac {s}{2t}+O(s^2).  
$$
Note also that the $O$-term is uniform as long as $t$ is bounded away from $0$. This proves the claim.
\end{proof}

\begin{ex}[see \cite{Ald,ChL}]\label{ex:upperhalf}
We denote  by $\Ha=\{t=(x,y)\in \mathbb R^2|\;y>0\}$ the upper half-plane. 
Let $\{B(x),x>0\}$ be the
standard Brownian motion. Then the field   $\{X(t),t=(x,y)\in \mathbb H\}$ of \textit{standardized Brownian motion increments} is defined  by
\begin{equation}\label{eq:defincr}
X(t)=\frac {B(x+y)-B(x)}{\sqrt {y}}
\end{equation}
is locally stationary with index $\alpha=1$. The local structure is given by
$$
C_{t}(s)=\left(|\,s_x|+|\,s_{x}+s_y|\right)/(2y),  
$$
where $t=(x,y)\in\Ha$ and  $s=(s_x,s_y)\in \R^2$.
\end{ex}
\begin{proof}
Let $t=(x,y)\in\mathbb H$ and $s=(s_x,s_y)\in \R^2$. Suppose first that $s_x>0$, $s_x+s_y>0$. Then
\begin{align*}
r(t,t+s)=&\Cov(X(t),X(t+s))=
\frac{y-s_x}{\sqrt{y(y+s_y)}}=1-\frac{s_x}{y}-\frac{s_y}{2y}+o(s_x,s_y)=\\
&1-(|\,s_x|+|\,s_x+s_y|)/(2y)+o(s_x,s_y).
\end{align*}
Now suppose that $s_x>0,s_x+s_y<0$. Then 
\begin{align*}
r(t,t+s)=&\Cov(X(t),X(t+s))=
\frac{y+s_y}{\sqrt{y(y+s_y)}}=1+\frac{s_y}{2y}+o(s_x,s_y)=\\&
1-(|\,s_x|+|\,s_x+s_y|)/(2y)+o(s_x,s_y).
\end{align*}
The remaining cases  can be treated analogously.
\end{proof}
Later, it will be convenient to have another representation of the field of standardized Brownian motion increments, which differs from~\eqref{eq:defincr} by a simple coordinate change. 
\begin{ex}\label{ex:upperhalf1}
Let $D=\{(x_1,x_2)\,|\,x_2>x_1\}$. Define a field $\{Y(t), t=(x_1,x_2)\in D\}$ by
$$
Y(x_1,x_2)=\frac{B(x_2)-B(x_1)}{\sqrt{x_2-x_1}}.
$$
Then the field $Y$ is locally stationary with $\alpha=1$. The local structure is given by
$$
C_{t}(s_1,s_2)=\frac{|\,s_1|+|\,s_2|}{\,2(x_2-x_1)},
$$
where $t=(x_1,x_2)\in D$ and $(s_1,s_2)\in \R^2$.
\end{ex}
\begin{ex}[see \cite{Ald}]
Let $\{B(t),t\in[0,1]\}$ be the Brownian bridge. Recall that the covariance function of $B$ is given by $\Cov(B(t_1),B(t_2))=\min(t_1,t_2)-t_1t_2$. Then the \textit{standardized Brownian bridge} $\{X(t),t\in(0,1)\}$ defined by
$$
X(t)=B(t)/\sqrt{t(1-t)}
$$
is locally stationary with index $\alpha=1$ and local structure $C_t(s)=\frac{|\,s|}{2t(1-t)}$.
\end{ex}
The next example is a multidimensional generalization of Example \ref{ex:upperhalf}.
\begin{ex}[see \cite{Ald}]
Let $\{\xi(A),A\in\mathcal B\}$ be a white noise on $(\R^d,\mathcal B,\textrm{Leb})$. This means that we are given a centered gaussian process $\xi$ indexed by the collection $\mathcal B$ of all Borel subsets of  $\R^d$ such that
$$
\Cov(\xi(A_1),\xi(A_2))=\textrm{Leb}(A_1\cap A_2) \quad \textrm{ for each } A_1,A_2\in\mathcal B, 
$$
where $\textrm{Leb}$ denotes the Lebesgue measure.
A set of the form 
$$
[x_1,y_1]\times\ldots\times[x_d,y_d] ,\qquad x_i<y_i,\,i=1,\ldots,d 
$$
is called \textit{rectangle}. Let
$$
\mathcal{R}=\{(x_1,y_1,\ldots,x_d,y_d)\in \R^{2d}\;|\;x_i<y_i,\,i=1,\ldots,d \}
$$
be the collection of all rectangles. Define a process $\{X(R),R\in\mathcal R\}$ indexed by rectangles by
$$
X(R)=\xi(R)/\sqrt{\textrm{Leb(R)}}.
$$
Then $X$ is locally stationary on $\mathcal R$ of index $\alpha=1$. The local structure is given by
$$
C_t(s)=\sum_{i=1}^d \left(|\,s_{ix}|+|\,s_{ix}+s_{iy}| \right)/(2y_i),
$$
where
$$
t=(x_1,y_1,\ldots,x_d,y_d)\in \mathcal R,\quad s=(s_{1x},s_{1y},\ldots,s_{dx},s_{dy})\in \R^{2d}. 
$$
\end{ex}
\begin{ex}
The Brownian motion with multidimensional time, introduced by L\'evy, is a centered gaussian process $\{B(t),t\in \R^d\}$ with the covariance function 
$$
\Cov(B(t),B(s))=\frac 1 2\left(\|\,t\|_2+\|\,s\|_2-\|\,t-s\|_2\right),
$$ 
where $\|\,t\|_2$ denotes the euclidian norm of $t$. 
Then the process $\{X(t),t\in \R^{d}\backslash \{0\}\}$ defined by 
$$
X(t)=B(t)/\sqrt{\|\,t\|_2} 
$$
is locally stationary with index $\alpha=1$. The local structure is given by 
$$
C_t(s)=\frac{\|\,s\|_2}{2\|\,t\|_2},\qquad t\in \R^d\backslash \{0\}, s\in \R^d.
$$
\end{ex}

To state the main theorems of this section, we need the following two definitions.
\begin{defi}
Let $\{X(t),t\in D\}$ be a gaussian field defined on some domain $D\subset \R^d$. Suppose that $X$ is  locally stationary with index $\alpha$ and local structure $C_t(s)$. For each $t\in D$, let $\{Y_t(s),s\in \R^d \}$ be a gaussian field defined by
\begin{equation}\label{eq:EY}
\E\left[Y_t(s)\right]=-C_t(s)
\end{equation}
and
\begin{equation}\label{eq:VarY}
\Cov (Y_t(s_1),Y_t(s_2))=C_t(s_1)+C_t(s_2)-C_t(s_1-s_2).
\end{equation}
Then $Y_t$ is called the tangent field of $X$ at the point $t$ conditioned on $X(t)=\infty$.
\end{defi}
The existence of $Y_t$ is guaranteed  by Proposition~\ref{prop:existtang}.  Moreover, the field $\tilde Y_t(s)=Y_t(s)+C_t(s)$ is $\alpha$-self-similar and has stationary increments. That is, for every $\lambda\in\R$, the field $\tilde Y_t(\lambda s)$ has the same finite-dimensional distributions as $|\,\lambda|^{\alpha}\tilde Y_t(s)$, and, for every $s_0\in \R^d$, the finite-dimensional distributions of the fields $\tilde Y_t(s_0+s)-\tilde Y_t(s_0)$ and $\tilde Y_t(s)$ coincide.
The next proposition, which will not be used in the sequel, may serve as a justification for the use of the term tangent field. 
\begin{proposition}
Assume that the assumptions of the previous  definition are satisfied. Let $q=q(u)=u^{-2/\alpha}$. For $t\in D$ and $u\in\R$, define a gaussian field $\{Y_t^u(s), s\in\R^d\}$ as the field $u(X(t+sq)-X(t))$ conditioned on $X(t)=u$. Then, for each fixed $t\in D$, the finite-dimensional distributions of $Y_t^u(s)$ converge, as $u\to\infty$, to the finite-dimensional distributions of $Y_t(s)$ from the previous definition.
\end{proposition}

\begin{defi}
With the above notation,
\begin{equation}\label{eq:defh}
H(t)=\lim_{T\to\infty} \frac{1}{T^d}\E\left[\exp\left(\sup_{s\in[0,T]^d}Y_t(s)\right) \right].
\end{equation}
is called the high excursion intensity of the field $X$.
\end{defi}
It was proved in~\cite{ChL} that $H(t)\in(0,\infty)$ exists and is continuous in $t$.
Alternatively, $H(t)$ can be defined by
\begin{equation}\label{eq:Ht}
H(t)=\lim_{T\to\infty} \frac 1{T^d} \int_{0}^{\infty}\P\left[\sup_{s\in[0,T]^d} Y_t(s)>w\right]e^wdw.
\end{equation}




The next theorem, proved in~\cite{ChL}, describes the asymptotic behavior of the high excursion probability of a locally stationary gaussian field. 
\begin{theorem}[see \cite{PM,ChL}]\label{theo:locstat}
Let $\{X(t),t\in D\}$ be a  gaussian field defined on some domain  $D\subset \mathbb R^d$. 
Suppose that $X$ is locally stationary of index $\alpha$ with local structure $C_t(s)$. Let $K\subset D$ be a
compact set with positive Jordan measure. Then, as $u\to\infty$,
$$
\P\left[\sup_{t\in K} X(t)>u \right]\sim  \frac 1 {\sqrt {2\pi}}\left(\int_K H(t)dt\right)  \;u^{\frac {2d}{\alpha}-1}e^{-u^2/2},
$$
where the function $H(t):D\to(0,\infty)$ is the high excursion intensity of $X$ defined in~\eqref{eq:defh}.
\end{theorem}

We are interested in the following special case of the above theorem.
\begin{cor}[see \cite{Ald,ChL}]\label{theo:tail}
Let $\{X(t),t\in\Ha\}$ be the field of standardized Brownian motion increments defined in Example~\ref{ex:upperhalf}. Let $K\subset \mathbb H$ be a compact set with positive Jordan measure. Then, as $u\to\infty$,  
$$
\P\left [ \sup_{t\in K} X(t) > u        \right ]\sim \frac 1 {4\sqrt {2\pi}} \int_K \frac {dx dy}{y^2} \ u^3 e^{- u^2/2}.
$$
\end{cor}

We also need the following theorem, which describes the asymptotic behavior of the high excursion probability over a finite grid with mesh size going to $0$.
\begin{theorem}\label{theo:locstatdiscr}
Suppose that  the conditions of Theorem~\ref{theo:locstat} are satisfied.
Let  $u\to+\infty$ and $q\to +0$ in such a way that $q u^{2/\alpha}\to a$ for some constant $a>0$. Then, as $u\to\infty$,
$$
\P\left[\sup_{t\in K\cap q\Z^d} X(t)>u \right]\sim  \frac 1 {\sqrt {2\pi}}\left(\int_K H_a(t)dt\right)  \;u^{\frac {2d}{\alpha}-1}e^{-u^2/2},
$$
where 
$$
H_a(t)=\lim_{T\to\infty} \frac{1}{T^d}\E\left[\exp\left(\sup_{s\in[0,T]^d\cap a\Z^d}Y_t(s)\right) \right].
$$
Furthermore, $\lim_{a\downarrow 0}H_a(t)=H(t)$, where $H(t)$ is the high excursion intensity of $X$.
\end{theorem}
We omit the proof of Theorem~\ref{theo:locstatdiscr}, since it is an adaptation of the proof of Lemma 12.2.4 from~\cite{Lead} to locally stationary fields.

\begin{cor}\label{cor:slep}
Let $\{X(t),t\in\R\}$ be the Slepian process defined in Example~\ref{ex:slepian}.  Let  $u\to+\infty$ and $q\to +0$ in such a way that $q u^2\to a$ for some constant $a>0$. Then
$$
\P\left[\sup_{t\in [0,1) \cap q \Z} X(t)>u \right]\sim  2F(2a) \frac 1 {\sqrt {2\pi}} \;u e^{-u^2/2},
$$
where
\begin{equation}\label{eq:deff}
F(a)=\lim_{T\to\infty} \frac 1 T \E \left[\exp \sup_{s\in [0,T]\cap a\Z} (B(s)-s/2)\right].
\end{equation}
Here, $\{B(s),s\geq 0\}$ is the standard Brownian motion. Further, $\lim_{a\downarrow 0} F(a)=1/2$.
\end{cor}
\begin{proof}
Actually, this was proved already in~\cite{P1}. According to Example~\ref{ex:slepian}, the Slepian process is locally stationary, the tangent process being $Y_t(s)=B(2s)-s$. It remains to use Theorem~\ref{theo:locstatdiscr}. See~\cite[Chapter 12]{Lead} for the proof that $\lim_{a\downarrow 0}F(a)=1/2$. 
\end{proof}
\begin{cor}\label{cor:1}
Let $\{X(t),t\in\Ha\}$ be the field of standardized Brownian motion increments defined in Example~\ref{ex:upperhalf}. Let $K\subset \mathbb H$ be a compact set with positive Jordan measure. Let  $u\to+\infty$ and $q\to +0$ in such a way that $q u^2\to a$ for some constant $a>0$. Then
$$
\P\left[\sup_{t\in K\cap q\Z^2} X(t)>u \right]\sim  \frac 1 {\sqrt {2\pi}}\left(\int_K  G(y) dxdy \right)  \;u^{3}e^{-u^2/2},
$$
where
\begin{equation}\label{eq:defg}
G(y)=\frac{1}{y^2}F\left(\frac{a}{y}\right)^2
\end{equation}
and the function $F$ is defined by~\eqref{eq:deff}.
Furthermore, we have $G(y)\sim 1/(4y^2)$ as $y\to+\infty$ and, for fixed $y$, $\lim_{a\to 0}G(y)=1/(4y^2)$.
\end{cor}
\begin{proof}
It is more convenient to use the notation of Example~\ref{ex:upperhalf1} rather than that of Example~\ref{ex:upperhalf}. Let $\{B_1(s),s\in\R\}$ and $\{B_2(s),s\in\R\}$ be two independent standard Brownian motions and let $W_1(s)=B_1(s)-s/2$, $W_2(s)=B_2(s)-s/2$.  The tangent process of $X$ is given, in the notation of Example~\ref{ex:upperhalf1}, by
$$
Y_{(x_1,x_2)}(s_1,s_2)=W_1\left(\frac{s_1}{x_2-x_1}\right)+W_2\left(\frac{s_2}{x_2-x_1}\right).
$$
Now we use Theorem~\ref{theo:locstatdiscr}. A simple change of variables shows that the high excursion intensity is given by
$$
H_a(t)= \frac{1}{(x_2-x_1)^2}\lim_{T\to\infty}\frac{1}{T^2}\E\left[\exp\sup_{(s_1,s_2)\in[0,T]^2\cap \frac{a}{x_2-x_1}\Z^2}\left(W_1(s_1)+W_2(s_2)\right) \right].
$$
Since the processes $W_1,W_2$ are independent, this is equal to
$$
\frac{1}{(x_2-x_1)^2}\lim_{T\to\infty} \frac{1}{T^2}\left(\E\left[\exp\sup_{s_1\in[0,T]\cap \frac{a}{x_2-x_1}\Z}W_1(s_1) \right]\right)^2,
$$
which is, by definition, $\frac{1}{(x_2-x_1)^2}F\left(\frac{a}{x_2-x_1}\right)^2$.
The lemma follows by switching to the notation of Example~\ref{ex:upperhalf}.
\end{proof}

\section{Standardized Brownian Motion Increments}\label{sec:proofbrown}
In this section we prove Theorem \ref{theo:main}. 
Let us describe briefly the method of the proof and  fix the notation.

Let $$\Ha=\{t=(x,y)\in\R^2\,|\,y>0\}$$ denote the open upper half-plane. 
A point $t=(x,y)\in \Ha$ will be often identified with the interval $[x,x+y]\subset \R$. There is a natural action of the group of affine transformations of the real line on $\Ha$ defined as follows.  If $g:x\mapsto ax+b$, where $a>0,b\in \R$, is an affine transformation of $\R$, then the action of $g$ on $\Ha$ is given by 
$$
g(t)= (ax+b,ay), \qquad t=(x,y)\in \Ha.
$$

Let $\{B(x),x\geq 0\}$ be the standard Brownian motion.
Recall that the random field $\{X(t),t=(x,y)\in\Ha\}$ of \textit{standardized Brownian motion increments} was defined in Example~\ref{ex:upperhalf} by
\begin{equation}\label{eq:XBrown}
X(t)=\frac{B(x+y)-B(x)}{y^{1/2}}.
\end{equation}
Note that the field $X$ is centered gaussian.  
For each $t\in\Ha$ the distribution of $X(t)$ is standard normal.

The following invariance property of the field $X$ will be useful
\begin{proposition}\label{prop:affine}
Let $g$ be an affine transformation of $\R$. Then, for each $t_1,\ldots,t_n\in\Ha$, the joint distribution of 
$X(g(t_1)),\ldots,X(g(t_n))$ coincides with the joint distribution of $X(t_1),\ldots,X(t_n)$.
\end{proposition}
The proof follows from the scaling property of the Brownian motion.

The above proposition allows us to state Theorem \ref{theo:main} in the following, equivalent form.
\begin{theorem}\label{theo:mainref}
For $n>1$ let $H(n)$ be the triangle 
$$
\{(x,y)\in \Ha\;|\;x\in[0,n],y\in[1,n-x]\}
$$ 
Define the random field $X$ by \eqref{eq:XBrown}. Then, for each $\tau\in\R$,
$$
\lim_{n\to\infty}\P\left [\sup_{t\in H(n)}X(t)\leq a_n+b_n\tau \right ]=\exp(-e^{-\tau}),
$$
where $a_n,b_n$ are constants defined by~\eqref{eq:defab}.
\end{theorem}  
The rest of the section is devoted to the proof of Theorem \ref{theo:mainref}.

Let $\tau\in \mathbb R$ be fixed.
Let $u_n=a_n+b_n\tau$ with $a_n, b_n$ defined by \eqref{eq:defab}. Note that $u_n\sim\sqrt{2\log n}$ as $n\to\infty$.
\begin{rem}\label{rem:tail}
We have, as $n\to \infty$, 
$$
\frac 1 {4\sqrt {2\pi}}u_n^3 e^{-u_n^2/2}\sim  e^{-\tau}/n. 
$$
\end{rem}

For $l>1$ define
$
H(n,l)=\{(x,y)\in H(n)\,|\, y\in[1,l]\}.
$
\begin{lem}\label{lem:l0}
The following holds for the high excursion probability over the triangle $H(n)\backslash H(n,l)$.
$$
\lim_{l\to\infty}\limsup_{n\to\infty}\P\left[\sup_{t\in H(n)\backslash H(n,l)} X(t)>u_n \right]=0. 
$$
\end{lem}
\begin{proof}
Divide $\mathbb H$ into rectangles
$$
R_{k,l}=\left[2^{l+1}k,2^{l+1}(k+1)\right]\times \left[2^l,2^{l+1}\right], \qquad k,l\in\mathbb Z.
$$
Note that all rectangles can be obtained from $R_{0,1}$ by the action of the one-dimensional affine group on $\Ha$. Thus, by the affine invariance of $X$ (Proposition~\ref{prop:affine}), 
the probability $\P\left[ \sup_{t\in R_{k,l}} X(t)>u_n\right ]$ is independent of $k,l$ and, by  Corollary~\ref{theo:tail} and Remark~\ref{rem:tail},
$$
\P\left[\sup_{t\in R_{k,l}} X(t)>u_n\right ]\sim \frac {e^{-\tau}} n \int_{R_{k,l}}\frac{dxdy}{y^2}=\frac {e^{-\tau}} n, \qquad n\to\infty.
$$
It is easy to see that $H(n)\backslash H(n,l)$ is covered by at most $\lceil 2n/l\rceil $ rectangles of the form $R_{k,l}$. Thus
$$
\limsup_{n\to\infty}\P\left[\sup_{t\in H(n)\backslash H(n,l)} X(t)>u_n\right ]\leq  \frac {2e^{-\tau}}{l} .
$$
The statement of the lemma follows.
\end{proof}

\begin{lem}\label{lem:l1}
We have
$$
\lim_{n\to\infty}\P\left[\sup_{t\in H(n,l)} X(t)\leq u_n\right]=\exp\left(-e^{-\tau}(l-1)/l\right).
$$
\end{lem}
\begin{proof}
Let
$$
H^*(n,l)= [0,n-1]\times[1,l],\qquad H_*(n,l)= [0,n-l]\times[1,l].
$$
Then $H_*(n,l)\subset H(n,l)\subset H^*(n,l)$.
So we have to prove that
\begin{equation}\label{eq:lem35}
\lim_{n\to\infty}\P\left[\sup_{t\in H^*(n,l)} X(t)\leq u_n\right]=\exp\left(-e^{-\tau}(l-1)/l\right).
\end{equation}
The same statement with $H_*(n,l)$ instead of $H^*(n,l)$ can be proved analogously and the lemma follows.

For $i=0,\ldots,n-2$ define $R_i=[i,i+1]\times[1,l]$. Then, by Corollary~\ref{theo:tail} and Remark~\ref{rem:tail},
\begin{equation}\label{eq:eq1}
\P\left[\sup_{t\in R_i} X(t)>u_n\right ]\sim \frac {e^{-\tau}} n \int_{R_i}\frac{dxdy}{y^2}=\frac {e^{-\tau}} n(l-1)/l , \qquad n\to\infty.
\end{equation}
Note, that by the affine invariance, the above probability is independent of~$i$. If the events ''$\sup_{t\in R_i} X(t)>u_n$'' were independent, we could finish the proof by applying the Poisson limit theorem. However, some additional work is required to overcome the dependence.

Fix $\eps,a>0$. Define $q_n=a/[2\log n]$ and
$$
R_i(\eps)=[i+\eps,i+1-\eps]\times[1,l],\qquad R_i(\eps,a)=R_{i}(\eps)\cap q_n\Z^2.
$$
Note that  $R_i(\eps,a)$ is a finite set depending on $n$. Let 
$$
H^*(n,l,\eps,a)=\bigcup_{i=0}^{n-2}R_i(\eps,a).
$$
\begin{lem}\label{lem:l1half}
Let
$$
\Delta_1(\eps,a)=\lim_{n\to\infty}n \, \P\left[\max_{t\in R_{0}(\eps,a)}X(t)> u_n\right]-e^{-\tau}(l-1)/l.
$$
Then $\lim_{a\downarrow 0}\lim_{\eps\downarrow 0}\Delta_1(a,\eps)=0$.
\end{lem}
\begin{proof}
Note that $\lim_{n\to\infty}q_nu_n^2=a$. 
We have, by Corollary~\ref{cor:1} and  Remark~\ref{rem:tail},
$$
\P\left[\sup_{t\in R_{0}(\eps,a)} X(t)>u_n \right]\sim  \left(\int_{R_{0}(\eps)}  4G(y) dxdy \right)e^{-\tau}/n ,\qquad n\to\infty.
$$
Here, the function $G$ is defined by~\eqref{eq:defg}.  Thus
$$
\Delta_1(\eps,a)=e^{-\tau}\left(\int_{R_{0}(\eps)}  4G(y) dxdy- (l-1)/l \right).
$$
Letting $\eps$ to $0$, we obtain
$$
\lim_{\eps\downarrow 0}\Delta_1(\eps,a)=e^{-\tau}\left(\int_{R_{0}}  4G(y) dxdy- (l-1)/l \right).
$$
To finish the proof note that $\lim_{a\to 0}G(y)=1/(4y^2)$ by Corollary~\ref{cor:1}.
\end{proof}
\begin{lem}\label{lem:l2}
We have
$$
\limsup_{n\to\infty}  \left(\P\left[\sup_{t\in H^*(n,l,\eps,a)} X(t)\leq u_n\right ]-\P\left[\sup_{t\in H^*(n,l)} X(t)\leq u_n\right ]\right)\leq \Delta_1(\eps,a),
$$
where  $\Delta_1(\eps,a)$ was defined in the previous lemma.
\end{lem}
\begin{proof}
We have, evidently,
\begin{align*}
&\P\left[\sup_{t\in H^*(n,l,\eps,a)} X(t)\leq u_n\right ]-\P\left[\sup_{t\in H^*(n,l)} X(t)\leq u_n\right ]=\\
&\P\left[ \sup_{t\in H^*(n,l)\backslash H^*(n,l,\eps,a)}X(t)>u_n \bigwedge \sup_{t\in H^*(n,l,\eps,a)} X(t)\leq u_n \right].
\end{align*}
The last probability is not greater than
\begin{align*}
&\sum_{i=0}^{n-2}\P\left[ \sup_{t\in R_i\backslash R_i(\eps,a)}X(t)>u_n \bigwedge \sup_{t\in R_i(\eps,a)} X(t)\leq u_n \right]=\\
&\sum_{i=0}^{n-2}\left(\P\left[ \sup_{t\in R_i}X(t)> u_n\right]-  \P\left[ \sup_{t\in R_i(\eps,a)}X(t)> u_n\right]\right)=\\
&(n-1)\P\left[ \sup_{t\in R_0}X(t)> u_n\right]- (n-1)\P\left[ \sup_{t\in R_0(\eps,a)}X(t)> u_n\right].
\end{align*}
To finish the proof it remains to use~\eqref{eq:eq1} for the first and  Lemma~\ref{lem:l1half} for the second term. 
\end{proof}

Let $\{Y(t),t\in H^*(n,l,\eps,a)\}$ be standard normal variables with the following covariance matrix:
\begin{align*}
&\E[Y(t_1)Y(t_2)]=\E[X(t_1)X(t_2)] &&\textrm{if } \exists i: t_1,t_2\in R_{i}(\eps,a),\\
&\E[Y(t_1)Y(t_2)]=0               &&\textrm{otherwise}. 
\end{align*}
Thus, we remove the dependence between $X(t_1)$ and $X(t_2)$ if $t_1$ and $t_2$ are in different $R_i$'s.

The next lemma is known as Berman's Inequality, see e.g. \cite[Theorem 4.2.1]{Lead}.
\begin{lem}
Suppose $\xi_1,\ldots,\xi_N$ are standard normal variables with covariance matrix $\Lambda^1=(\Lambda_{ij}^1)$, and $\eta_1,\ldots,\eta_N$ similarly 
with covariance matrix $\Lambda^2=(\Lambda_{ij}^2)$, and let 
$\rho_{ij}=\max (|\Lambda_{ij}^1|,|\Lambda_{ij}^2|)$. Then 
\begin{align*}
\P\left[\max_{1\leq i\leq N}\xi_i\leq u\right]-&\P\left[\max_{1\leq i\leq N}\eta_i\leq u\right] \leq \\ 
&\frac 1 {2\pi} \sum_{1\leq i<j\leq N}|\Lambda^1_{ij}-\Lambda^2_{ij}|(1-\rho_{ij}^2)^{-1/2}\exp\left(-\frac{u^2}{1+\rho_{ij}}\right). 
\end{align*}
\end{lem}
The next lemma shows that the high excursion behavior of the gaussian vector $X(t)$ coincides with that of $Y(t)$.
\begin{lem}\label{lem:l3}
We have, for fixed $\eps$ and $a$,
$$
\lim_{n\to\infty}  \left(\P\left[\sup_{t\in H^*(n,l,\eps,a)} X(t)\leq u_n\right ]- \P\left[\sup_{t\in H^*(n,l,\eps,a)} Y(t)\leq u_n\right ]\right)=0
$$
\end{lem}
\begin{proof}
We are going to use Berman's Inequality for the variables $\{X(t),t\in H^*(n,l,\eps,a)\}$ and $\{Y(t),t\in H^*(n,l,\eps,a)\}$. 
Let us write $t_1\sim t_2$ if $t_1$ and $t_2$ are contained in the same set $R_i(\eps,a)$. 
Define $\Lambda^X_{t_1,t_2}=\E[X(t_1)X(t_2)]$, $\Lambda^Y_{t_1,t_2}=\E[Y(t_1)Y(t_2)]$ and $\rho_{t_1t_2}=\max(\Lambda^X_{t_1,t_2},\Lambda^Y_{t_1,t_2})$. Then $\Lambda^X_{t_1,t_2}=\Lambda^Y_{t_1,t_2}$ if $t_1\sim t_2$. It follows that
$$
\Lambda_{t_1t_2}^X-\Lambda_{t_1t_2}^Y=\begin{cases}0,&\textrm{ if } t_1\sim t_2\\ \Lambda_{t_1t_2}^X,& \textrm{ else.}\end{cases}
$$ 
It is easy to see that the correlations $\Lambda^X_{t_1,t_2}$, $t_1\nsim t_2$ are bounded away from $1$ by some constant depending on $\eps$ but not on $n$. Thus, we have $\Lambda^X_{t_1,t_2}\leq \delta<1$ provided that $t_1\nsim t_2$. Using Berman Inequality we obtain
\begin{align*}
\P&\left[ \max_{t\in H^*(n,l,\eps,a)}X(t) \leq u_n\right]-\P\left[\max_{t\in H^*(n,l,\eps,a)}Y(t)\leq u_n\right]\leq\\
&\frac 1 {4\pi} \sum_{\genfrac{}{}{0pt}{1} {t_1,t_2\in H^*(n,l,\eps,a)} {t_1\neq t_2}} |\Lambda^X_{t_1t_2}-\Lambda^Y_{t_1t_2}|(1-\rho_{t_1t_2}^2)^{-1/2}\exp\left(-u_n^2/(1+\rho_{t_1t_2})\right).
\end{align*}
The right-hand side is not greater than
$$
\frac 1 {4\pi} \sum_{\genfrac{}{}{0pt}{1} {t_1,t_2\in H^*(n,l,\eps,a)} {t_1\nsim t_2}} \Lambda^X_{t_1t_2}(1-\delta^2)^{-1/2}\exp\left(-u_n^2/(1+\delta)\right),
$$
which is smaller than
$$
K\exp\left(-u_n^2/(1+\delta)\right) \sum_{\genfrac{}{}{0pt}{1} {t_1,t_2\in H^*(n,l,\eps,a)} {t_1\nsim t_2}} \Lambda^X_{t_1t_2}.
$$
for some constant $K$ depending on $\eps$ but not on $n$.

Recall that $R_i(\eps,a)=q_n\Z^2\cap R_i(\eps)$. It follows that the number of elements of $R_{i}(\eps,a)$ is less than $O(\log^2 n)$, where the constant in the $O$-term depends only on $a$ and $l$. 

It is easy to see that   $X(t_1)$ and $X(t_2)$ are independent provided that $t_1\in R_{i_1}$ and $t_2\in R_{i_2}$ with $|i_1-i_2|>l+1$. Consequently, the number of pairs $(t_1,t_2)$ such that $X(t_1)$ and $X(t_2)$ are dependent is less than $O(n\log^4 n)$. Thus
$$
\P\left[\max_{t\in H(n,\eps,a)}X(t)\leq u_n\right]-\P\left[\max_{t\in H(n,\eps,a)}Y(t)\leq u_n\right]\leq K' n(\log^4 n) e^{-u_n^2/(1+\delta)}.
$$
where $K'$ depends on $\eps$ and $a$, but not on $n$. Recall that $u_n\sim \sqrt{2\log n}$. 
The statement of the lemma follows.
\end{proof}

\begin{lem}\label{lem:l4}
Let
$$
\Delta_2(\eps,a)=\limsup_{n\to \infty}\left| \P\left[\max_{t\in H^*(n,l,\eps,a)}Y(t)\leq u_n\right]-\exp(-e^{-\tau}(l-1)/l)\right|.
$$
Then $\lim_{a\downarrow 0}\lim_{\eps\downarrow 0} \Delta_2(\eps,a)=0$.
\end{lem}
\begin{proof} 
Since $Y(t_1)$ and $Y(t_2)$ are independent if $t_1$ and $t_2$ are in different $R_i$'s, we have
\begin{align*}
\P\left[\max_{t\in H^*(n,l,\eps,a)}Y(t)\leq u_n\right]=& \left(1-\P\left[\max_{t\in R_0(\eps,a)}Y(t)> u_n\right]\right)^{n-1}=\\ 
&\left(1-\P\left[\max_{t\in R_0(\eps,a)}X(t)> u_n\right]\right)^{n-1}.
\end{align*}
Using  this and Lemma~\ref{lem:l1half}, we obtain
$$
\lim_{n\to\infty}\P\left[\max_{t\in H^*(n,l,\eps,a)}Y(t)\leq u_n\right]= \exp(-e^{-\tau}(l-1)/l+\Delta_1(\eps,a)),
$$
where $\lim_{a\downarrow 0}\lim_{\eps\downarrow 0} \Delta_1(\eps,a)=0$.
This proves Lemma~\ref{lem:l4}.
\end{proof}
Now we are able to finish the proof of Lemma~\ref{lem:l1}. Recall that we have to prove~\eqref{eq:lem35}.
Using Lemmas~\ref{lem:l3} and~\ref{lem:l4}, we obtain
$$
\limsup_{n\to \infty}\left| \P\left[\max_{t\in H^*(n,l,\eps,a)}X(t)\leq u_n\right]-\exp(-e^{-\tau}(l-1)/l)\right|=\Delta_2(\eps,a).
$$
Now use Lemma~\ref{lem:l2} to obtain
$$
\limsup_{n\to \infty}\left| \P\left[\max_{t\in H^*(n,l)}X(t)\leq u_n\right]-\exp(-e^{-\tau}(l-1)/l)\right|\leq \Delta_1(\eps,a)+\Delta_2(\eps,a).
$$
To finish the proof let $\eps,a\downarrow 0$.
\end{proof}
\begin{proof}[\textbf{Proof of Theorem~\ref{theo:mainref}.}]
It follows from $H(n,l)\subset H(n)$ that 
$$
\limsup_{n\to\infty}\P\left[\sup_{t\in H(n)} X(t)\leq u_n\right]\leq \lim_{n\to\infty}\P\left[\sup_{t\in H(n,l)} X(t)\leq u_n\right], 
$$
which is equal to $\exp(-e^{-\tau}(l-1)/l)$ by Lemma~\ref{lem:l1}. Letting $l\to\infty$ we obtain
$$
\limsup_{n\to\infty}\P\left[\sup_{t\in H(n)} X(t)\leq u_n\right]\leq \exp(-e^{-\tau}).
$$
On the other hand, we have
$$
\P\left[\sup_{t\in H(n)} X(t)\leq u_n\right]\geq \P\left[\sup_{t\in H(n,l)} X(t)\leq u_n\right]-\P\left[\sup_{t\in H(n)\backslash H(n,l)} X(t)>u_n \right].
$$
Letting $n\to\infty$, $l\to\infty$ and using Lemma~\ref{lem:l1} for the first and Lemma~\ref{lem:l0} for the second term, we obtain
$$
\liminf_{n\to\infty}\P\left[\sup_{t\in H(n)} X(t)\leq u_n\right]\geq \exp(-e^{-\tau}),
$$
which finishes the proof of Theorem~\ref{theo:mainref}.
\end{proof}

\section{Distributional Convergence in  the Erd\"os-Renyi Law}\label{sec:prooferdren}
In this section we sketch a proof of Theorem~\ref{theo:mainerdren}.
 
Let $\{X(t),t\in \R\}$ be the Slepian process, i.e. the stationary gaussian process defined by $X(t)=\int_{t}^{t+1}dW$, where $dW$ is the white noise on $\R$. Equivalently, $X$ can be defined as a stationary gaussian process with the covariance function  given by 
$$
\Cov(X(0),X(t))=\begin{cases}1-|t|, & \textrm {if } |t|\leq 1,\\0, & \textrm{otherwise}.\end{cases}
$$
Let $c$ be a positive constant and define $l_n=[c\log n]$. Let $q_n=1/l_n$. Finally, fix $\tau\in \R$ and let
$$
u_n=\sqrt{2\log n}+\frac {-1/2 \log\log  n+\log (2F(4/c)/(c\sqrt{\pi}))+\tau}{\sqrt{2 \log n}},
$$
where the function $F$ is defined by~\eqref{eq:deff}.

It is easy to see that the random variables $\{X(kq_n), k=0,\ldots,n-l_n\}$ have the same joint law as $\{(S_{k+l_n}-S_k)/\sqrt{l_n}, k=0,\ldots, n-l_n\}$.
It follows from Corollary~\ref{cor:slep} with $a=\lim_{n\to\infty}q_n u_n^2=2/c$ that
$$
\P\left[\max_{k=0,\ldots, l_n-1 } X(kq_n)>u_n\right]\sim  \frac1{\sqrt{2\pi}} u_n e^{-{u_n}^2/2}2F(4/c)\sim \frac{l_n}{n}e^{-\tau}.
$$

Now we would like to apply the Poisson limit theorem to the events 
$$
\max_{k=ml_n,\ldots, (m+1)l_n-1 } X(kq_n)>u_n, \qquad m=0,\ldots, n/l_n-1.
$$ 
To prove the approximate independence of the above events, one can use Berman Inequality as it was done in Lemma~\ref{lem:l3}. We omit the details.  
Thus, by the Poisson limit theorem, 
$$
\lim_{n\to\infty}\P\left[\max_{k=0,\ldots, n-l_n} X(kq_n)>u_n\right]=\lim_{n\to\infty}\left(1-\frac{l_n}{n}e^{-\tau}\right)^{n/l_n-1}=\exp(-e^{-\tau}).
$$
This proves Theorem~\ref{theo:mainerdren}.

\section{Standardized Increments of the Gaussian Random Walk}\label{sec:proofdiscr}
In this section we prove Theorem~\ref{theo:maindiscr}.

First we introduce some notation. Let $\tau\in\R$ be fixed. Define
$$
u_n=\sqrt{2\log n}+\frac{1/2\log\log n-\log(2^{-1}\sqrt {\pi})+\tau}{\sqrt{2\log n}}
$$
and let $q_n=1/[\log n]$. Note that $\lim_{n\to\infty}q_nu_n^2=2$.
\begin{rem}\label{rem:discr1}
We have, as $n\to\infty$,
$$
\frac{1}{\sqrt {2\pi}}u_n^3e^{-u_n^2/2}\sim \frac{\log n}{n}e^{-\tau}.
$$
\end{rem}
Let $\{B(x),x\geq 0\}$ be the standard Brownian motion. Recall that $\Ha$ denotes the upper half-plane and that the random field  of standardized Brownian increments $\{X(t),t=(x,y)\in\Ha\}$ was defined in Example~\ref{ex:upperhalf}  by 
$$
X(x,y)=\frac{B(x+y)-B(x)}{\sqrt y}.
$$

Let 
$$
T(n)=\{(xq_n,yq_n)\,|\, x=0,\ldots,n; y=1,\ldots,n-x\}.
$$
Then it is easy to see that the random vector $\{ (S_j-S_i)/\sqrt{j-i}, 0\leq i<j\leq n\}$ has the same distribution as $\{X(t), t\in T(n) \}$. 
Thus, our aim is to prove that
\begin{equation}\label{eq:theo13}
\lim_{n\to\infty}\P\left[\sup_{t\in T(n)}X(t)\leq u_n \right]=\exp\left(-e^{-\tau}\int_{0}^{\infty} G(y)dy\right).
\end{equation}
Here, $G$ is defined by~\eqref{eq:defg1} or, equivalently, by~\eqref{eq:defg} with $a=2$.  
First, we prove that the integral $\int_{0}^{\infty} G(y)dy$ is finite. 

\begin{lemma}
$\int_{0}^{\infty} G(y)dy$ is finite.
\end{lemma}
\begin{proof}
Since $G(y)\sim 1/(4y^2)$ as $y\to\infty$ by Corollary~\ref{cor:1}, we have only to prove that $\int_{0}^{1}G(y)dy$ is finite. 

Fix some $0<l<1$. Let $K=[0,1]\times[l,1]$. Again using Corollary~\ref{cor:1} and Remark~\ref{rem:discr1} we obtain, as $n\to\infty$,
$$
\P\left[\sup_{t\in T(n)\cap K} X(t)>u_n \right]\sim \left(\int_{l}^{1}G(y)dy\right) \frac{\log n}{n}e^{-\tau}. 
$$
On the other hand, since $T(n)\cap K$ consists of at most $\log^2 n$ points, we have, evidently,
$$
\P\left[\sup_{t\in T(n)\cap K} X(t)>u_n \right]\leq (\log^2n)(1-\Phi(u_n)),
$$ 
where $\Phi$ is the standard normal distribution function. Using that $1-\Phi(u)\sim \frac1{\sqrt {2\pi}}\frac 1 u e^{-u^2/2}$ as $u\to\infty$, as well as Remark~\ref{rem:discr1}, we obtain that the right-hand side is asymptotically equivalent to 
$$
(\log^2 n)\frac{1}{\sqrt{2\pi}}\frac{1}{u_n}e^{-u_n^2/2}\sim  \frac{1}{4}\frac{\log n}{n}e^{-\tau}.
$$
It follows that $\int_{l}^{1}G(y)dy\leq 1/4$ for all $l>0$, which proves the lemma.
\end{proof}
For $0\leq l_1<l_2\leq \infty$ define
$$
T(n,l_1,l_2)=T(n)\cap\{(x,y)\in\Ha\,|\,y\in(l_1,l_2)\}.
$$

\begin{lem}\label{lem:l0discr}
We have
$$
\lim_{l_1\to 0}\limsup_{n\to\infty}\P\left[\max_{t\in T(n,0,l_1)}X(t)>u_n\right]=0.
$$
\end{lem}
\begin{proof}
The number 
of elements in the finite set $T(n,0,l_1)$ does not exceed $l_1 n\log n$. 
We have, as $n\to\infty$,
$$
\P\left[\max_{t\in T(n,0,l_1)}X(t)>u_n\right]\leq l_1 n(\log n) (1-\Phi(u_n))\sim l_1 n(\log n) \frac{1}{\sqrt{2\pi}}\frac{1}{u_n}e^{-u_n^2/2}.
$$
Using Remark~\ref{rem:discr1}, we obtain
$$
\limsup_{n\to\infty}\P\left[\max_{t\in T(n,0,l_1)}X(t)>u_n\right]\leq \frac 1 4 e^{-\tau}l_1.
$$
This finishes the proof.
\end{proof}

\begin{lem}\label{lem:l0discr1}
We have
$$
\lim_{l_2\to+\infty}\limsup_{n\to\infty}\P\left[\max_{t\in T(n,l_2,+\infty)}X(t)>u_n\right]=0.
$$
\end{lem}
\begin{proof}
The proof is analogous to the proof of Lemma~\ref{lem:l0} and is therefore omitted.  
\end{proof}

\begin{lem}\label{lem:l1discr}
We have
$$
\lim_{n\to\infty}\P\left[\sup_{t\in T(n,l_1,l_2)} X(t)\leq u_n\right]=
\exp\left(-e^{-\tau}\int_{l_1}^{\l_2}G(y)dy\right).
$$
\end{lem}
\begin{proof}
Let
$$
T^*(n,l_1,l_2)=q_n\Z^2\cap ([0,\lceil nq_n-l_1\rceil ]\times[l_1,l_2]),
$$
$$
T_*(n,l_1,l_2)= q_n\Z^2\cap([0,\lfloor nq_n-l_2\rfloor]\times[l_1,l_2]).
$$
Then $T_*(n,l_1,l_2)\subset T(n,l_1,l_2)\subset T^*(n,l_1,l_2)$.
Thus, to prove Lemma~\ref{lem:l1discr} we have to show that
\begin{equation}\label{eq:discr}
\lim_{n\to\infty}\P\left[\sup_{t\in T^*(n,l_1,l_2)} X(t)\leq u_n\right]=\exp\left(-e^{-\tau}\int_{l_1}^{l_2}G(y)dy\right),
\end{equation}
since the proof of the corresponding statement with  $T_*(n,l_1,l_2)$ instead of $T^*(n,l_1,l_2)$ is analogous.

For $i=0,\ldots,\lceil nq_n-l_1\rceil-1$ define 
$$
R_i=q_n\Z^2\cap ([i,i+1]\times[l_1,l_2]).
$$ 
Recall that $\lim_{n\to\infty}q_n u_n^2=2$. Then, by Corollary~\ref{cor:1} with $a=2$ and Remark~\ref{rem:discr1},
\begin{equation}\label{eq:eq1discr}
\P\left[\sup_{t\in R_i} X(t)>u_n\right ]\sim e^{-\tau}\frac {\log n} n\int_{l_1}^{l_2}G(y)dy , \qquad n\to\infty.
\end{equation}
By the affine invariance (Proposition~\ref{prop:affine}), the above probability is independent of~$i$. 
As in the previous section, the difficulty is the dependence of the events ''$\sup_{t\in R_i} X(t)>u_n$''. If the events were independent, we were done by the Poisson limit theorem.
Fix $\eps>0$. Define
$$
R_i(\eps)=q_n\Z^2\cap ([i+\eps,i+1-\eps]\times[l_1,l_2]).
$$
and
$$
T^*(n,l_1,l_2,\eps)=\bigcup_{i=0}^{\lceil nq_n-l_1\rceil-1} R_i(\eps).
$$
Note that the finite set $R_i(\eps)$ depends on $n$.

\begin{lem}\label{lem:l333}
We have
$$
0\leq \limsup_{n\to\infty} \left(\P\left[\max_{t\in T^*(n,l_1,l_2,\eps)}X(t)\leq u_n \right]-\P\left[\max_{t\in T^*(n,l_1,l_2)}X(t)\leq u_n \right]\right)<c_1 \eps.
$$
for some constant $c_1$ depending only on $l_1,l_2$.
\end{lem}
\begin{proof}
Proceeding as in Lemma~\ref{lem:l2}, we obtain
\begin{align*}
&\P\left[\max_{t\in T^*(n,l_1,l_2,\eps)}X(t)\leq u_n \right]-\P\left[\max_{t\in T^*(n,l_1,l_2)}X(t)\leq u_n \right]\leq\\
&(\lceil nq_n-l_1\rceil-1)\left(\P\left[\max_{t\in R_0} X(t)> u_n \right]-\P\left[\max_{t\in R_0(\eps)} X(t)> u_n \right]\right).
\end{align*}
By Corollary~\ref{cor:1} with $a=2$ and Remark~\ref{rem:discr1}
\begin{equation}\label{eq:eq444}
\P\left[\max_{t\in R_0(\eps)} X(t)> u_n \right]\sim   e^{-\tau}\frac {\log n} n (1-2\eps)\int_{l_1}^{l_2}G(y)dy   ,\qquad n\to\infty.
\end{equation}
Using this together with~\eqref{eq:eq1discr}, we obtain the statement of the lemma.
\end{proof}

Let  $\{Y(t), t\in T^*(n,l_1,l_2,\eps) \}$ be a gaussian vector with the following covariance structure
\begin{align*}
&\E[Y(t_1)Y(t_2)]=\E[X(t_1)X(t_2)] &&\textrm{if } \exists i: t_1,t_2\in R_{i}(\eps),\\
&\E[Y(t_1)Y(t_2)]=0               &&\textrm{otherwise}. 
\end{align*}
Thus, we remove the dependence between $X(t_1)$ and $X(t_2)$ if $t_1$ and $t_2$ are in different $R_{i}(\eps)$'s.
\begin{lem}\label{lem:l334}
We have
$$
\lim_{n\to\infty}\left(\P\left[\max_{t\in T^*(n,l_1,l_2,\eps)} X(t)\leq u_n \right]-\P\left[\max_{t\in T^*(n,l_1,l_2,\eps)} Y(t)\leq u_n \right]\right)= 0.
$$
\end{lem}
\begin{proof}
The proof, which we omit,  uses Berman's inequality and is analogous to the proof of Lemma~\ref{lem:l3}.
\end{proof}

\begin{lem}\label{lem:l335}
We have
$$
\lim_{\eps\downarrow 0}\lim_{n\to\infty} \P\left[\max_{t\in T^*(n,l_1,l_2,\eps)} Y(t)\leq u_n \right]= \exp\left(-e^{-\tau}\int_{l_1}^{l_2}G(y)dy\right)
$$
\end{lem}
\begin{proof}
Since $Y(t_1)$ and $Y(t_2)$ are independent provided that $t_1$ and $t_2$ are in different $R_i(\eps)$'s, we have
\begin{align*}
\P\left[\max_{t\in T^*(n,l_1,l_2,\eps)} Y(t)\leq u_n \right]=& \left(1-\P\left[\max_{t\in R_0(\eps)} Y(t)> u_n \right]\right)^{\lceil nq_n-l_1\rceil-1}=\\
&\left(1-\P\left[\max_{t\in R_0(\eps)} X(t)> u_n \right]\right)^{\lceil nq_n-l_1\rceil-1}.
\end{align*}
Recall that $\lceil nq_n-l_1\rceil-1\sim n/\log n$, $n\to\infty$. Using~\eqref{eq:eq444}, we obtain
$$
\lim_{n\to\infty}\P\left[\max_{t\in T^*(n,l_1,l_2,\eps)} Y(t)\leq u_n \right]= \exp\left(-e^{-\tau}(1-2\eps)\int_{l_1}^{l_2}G(y)dy\right)
$$
and the lemma follows by letting $\eps\downarrow 0$.
\end{proof}
Now we can finish the proof of Lemma~\ref{lem:l1discr}. We have to show~\eqref{eq:discr}. But it follows easily from Lemmas~\ref{lem:l333},~\ref{lem:l334} and~\ref{lem:l335}.
\end{proof}
\begin{proof}[\textbf{Proof of Theorem~\ref{theo:maindiscr}.}] Recall that we have to prove~\eqref{eq:theo13}.
The evident  inequality
$$
\P\left[\max_{t\in T(n)}X(t) \leq u_n\right] \leq \P\left[\max_{t\in T(n,l_1,l_2)}X(t) \leq u_n\right]\ 
$$
together with Lemma~\ref{lem:l1discr} imply that
$$
\limsup_{n\to\infty}\P\left[\max_{t\in T(n)}X(t) \leq u_n\right]\leq \exp\left(-e^{-\tau}\int_{0}^{\infty} G(y)dy\right).
$$
Now, using  Lemmas \ref{lem:l1discr},\,\ref{lem:l0discr},\,\ref{lem:l0discr1} and the inequality
\begin{align*}
\P\left[\max_{t\in T(n)}X(t) \leq u_n\right]&\geq \P\left[\max_{t\in T(n,l_1,l_2)}X(t)\leq u_n \right]-\\
&\P\left[\max_{t\in T(n,0,l_1)}X(t)>u_n\right]-\P\left[\max_{t\in T(n,l_2,+\infty)}X(t)>u_n\right]  
\end{align*}
we obtain, by letting $l_1\to 0$ and $l_2\to\infty$,
$$
\liminf_{n\to\infty}\P\left[\max_{t\in T(n)}X(t) \leq u_n\right]\geq \exp\left(-e^{-\tau}\int_{0}^{\infty} G(y)dy\right).
$$
This finishes the proof of Theorem~\ref{theo:maindiscr}.
\end{proof}

\section{An Explicit Formula for the Constant $H$}\label{sec:const_H}
Let $\{\xi_i\}_{i=1}^{\infty}$ be a sequence of i.i.d. standard gaussian variables. Let $S_n=\sum_{i=1}^n\xi_i$, $S_0=0$ be the gaussian random walk  and recall that the maximum of standardized gaussian random walk increments was defined  by
$$
L_n=\max_{0\leq i<j\leq n}\frac{S_j-S_i}{\sqrt{j-i\,}}.
$$
It was shown in Theorem~\ref{theo:maindiscr} that the extreme-value rate as $n\to\infty$ of $L_n$ is $Hn\log n$. Here, $H>0$ is a constant which was defined as follows. 
Let $\{B(t),t\geq 0\}$ be the standard Brownian motion. Let
$$
F(a)=\lim_{T\to\infty} \frac 1 T\E \left[\exp \sup_{t\in [0,T]\cap a\Z} (B(t)-t/2)\right]
$$
and
$$
G(y)=\frac{1}{y^2}F\left(\frac{2}{y}\right)^2.
$$
Then $H=4\int_{0}^{\infty} G(y)dy$. 
This formulae do not allow to calculate the constant $H$ numerically. Our goal is to obtain a different representation of $H$ which makes numerical calculations possible. 
\begin{theo}\label{theo:main_H}
Let $\Phi$ be the standard normal distribution function. We have
$$
H=\int_{0}^{\infty}\exp\left\{-4 \sum_{k=1}^{\infty}\frac 1k  \Phi(-\sqrt{k/(2y)})\right\}dy.
$$
\end{theo}
A numerical calculation shows that $H\approx 0.21$.
The rest of the section is devoted to the proof of the above theorem.

Fix some $a>0$. Let $\{X_i,i=1,\ldots\}$ be i.i.d. gaussian random variables with $\E X_i=-a/2$, $\Var X_i=a$. Define the negatively drifted gaussian random walk $Z_n=\sum_{i=1}^n X_i$, $Z_0=0$. Note that $Z_n$ drifts to $-\infty$ a.s. The behavior of one-dimensional random walks is well-studied , see~\cite[Chapters XII and XVIII]{Feller}, \cite{Spitzer} as well as~\cite{JvL} for the drifted gaussian case.
  
Let $p_{\infty}(a)=\P[Z_n<0 \;\forall n\in\N]$ be the probability that $Z_n$ never enters the upper half-line. By Spitzers Identity
$$
p_{\infty}(a)=\exp\left\{-\sum_{k=1}^{\infty}\frac 1 k \P(Z_k>0)\right\}.
$$
Theorem~\ref{theo:main_H} is then easily seen to follow from
\begin{theo}
We have
$
F(a)=p_{\infty}^2(a)/a.
$
\end{theo}
\begin{proof} Let $M_n=\max_{i=0,\ldots,n} Z_i$. It is easy to see from the definition of $F$ that 
$$
F(a)=\frac 1 a \lim_{n\to\infty}\frac 1 n \E[e^{M_n}].
$$
Thus, we concentrate on the calculation of the above limit. 
Let 
$$
g(w)=\sum_{n=0}^{\infty}w^n\E[e^{M_n}].
$$
By~\cite{Spitzer}, equation $(1)$ on page $207$, we have
$$
g(w)=(1-w)^{-1}\exp\left\{-\sum_{k=1}^{\infty} \frac{w^k}{k}\E[(1-e^{Z_k}) 1_{Z_k>0}]\right\}.
$$
Now, recalling that $Z_k\sim \mathcal N(-ak/2,ak)$,
\begin{align*}
\E [e^{Z_k}1_{Z_k>0}]&=\frac{1}{\sqrt{2\pi}}\int_{\sqrt{ak}/2}^{\infty}e^{\sqrt{ak}x -ak/2}e^{-x^2/2}dx=\frac{1}{\sqrt{2\pi}}\int_{\sqrt{ak}/2}^{\infty}e^{-(x-\sqrt{ak})^2/2}dx\\
&=\frac{1}{\sqrt{2\pi}}\int_{-\sqrt{ak}/2}^{\infty}e^{-x^2/2}dx=1-\P[Z_k>0].
\end{align*}
Thus, we obtain
\begin{align*}
g(w)&=(1-w)^{-1}\exp\left\{-\sum_{k=1}^{\infty} \frac{w^k}{k}(2\P[Z_k>0]-1)\right\}\\
&=(1-w)^{-2}\exp\left\{-2\sum_{k=1}^{\infty} \frac{w^k}{k}\P[Z_k>0]\right\}
\end{align*}
and, consequently,
$$
g(w)\sim \frac{p_{\infty}^2(a)}{(1-w)^2}\qquad\textrm{ as } w\uparrow 1.
$$
By a well-known Tauberian Theorem (see e.g.~\cite[Theorem 5 on p. 447]{Feller}) it follows that
$$
\lim_{n\to\infty}\frac 1 n \E[e^{M_n}]=p_{\infty}^2(a).
$$
This finishes the proof.
\end{proof}

\textbf{Acknowledgements.} 
The author is grateful to M.Denker, A.Munk and M.Schlather  for their support and encouragement.

\end{document}